\documentclass[letterpaper, 10 pt, conference]{ieeeconf}  

\IEEEoverridecommandlockouts                              
\usepackage{cite}
\usepackage{amsmath,amssymb,amsfonts}
\usepackage{bbm}
\usepackage{mathrsfs}
\usepackage{algorithm}
\usepackage{algpseudocode}
\usepackage{array}
\usepackage{graphicx}
\usepackage{multicol}
\usepackage{multirow}
\usepackage{epstopdf}
\usepackage{makecell}
\usepackage{textcomp}
\usepackage{booktabs}
\usepackage{cases}
\usepackage{soul}
\usepackage{subfigure}
\usepackage{url}
\usepackage{comment}

\newcommand{\argmin}{\operatorname{arg}\min}

\newtheorem{thm}{Theorem}

\newtheorem{rem}{Remark}

\title{\LARGE \bf
Accelerated Nonconvex ADMM with Self-Adaptive Penalty for Rank-Constrained Model Identification
}

\author{Qingyuan Liu, Zhengchao Huang, Hao Ye, Dexian Huang, and Chao Shang, \IEEEmembership{Member, IEEE}
\thanks{*This work was supported by National Natural Science Foundation of China under Grant 62003187. \textit{(Corresponding author: C. Shang)}}
\thanks{Q. Liu and Z. Huang are with Department of Automation, Tsinghua University, Beijing 100084, China. E-mail: {\tt\small \{lqy20,huang-zc19\}@mails.tsinghua.edu.cn}}%
\thanks{H. Ye, D. Huang, and C. Shang are with Department of Automation, BNRist, Tsinghua University, Beijing 100084, China. E-mail: {\tt\small \{haoye,huangdx,c-shang\}@tsinghua.edu.cn}}%
}

\begin{document}
\renewcommand{\algorithmicrequire}{\textbf{Input:}}
\renewcommand{\algorithmicensure}{\textbf{Output:}}
\newcommand{\bZ}{\mathbf{Z}}
\newcommand{\bz}{\mathbf{z}}
\newcommand{\be}{\mathbf{e}}
\newcommand{\tbz}{\tilde{\mathbf{z}}}
\newcommand{\bLambda}{\mathbf{\Lambda}}
\newcommand{\blambda}{\boldsymbol{\lambda}}
\newcommand{\bmu}{\boldsymbol{\mu}}
\newcommand{\Lb}{\mathcal{L}_\beta}
\newcommand{\bxi}{\boldsymbol{\xi}}
\newcommand{\bPhi}{\mathbf{\Phi}}
\newcommand{\btheta}{\boldsymbol{\theta}}

\maketitle
\thispagestyle{empty}
\pagestyle{empty}

\begin{abstract}
The alternating direction method of multipliers (ADMM) has been widely adopted in low-rank approximation and low-order model identification tasks; however, the performance of nonconvex ADMM is highly reliant on the choice of penalty parameter. To accelerate ADMM for solving rank-constrained identification problems, this paper proposes a new self-adaptive strategy for automatic penalty update. Guided by first-order analysis of the increment of the augmented Lagrangian, the self-adaptive penalty updating enables effective and balanced minimization of both primal and dual residuals and thus ensures a stable convergence. Moreover, improved efficiency can be obtained within the Anderson acceleration scheme. Numerical examples show that the proposed strategy significantly accelerates the convergence of nonconvex ADMM while alleviating the critical reliance on tedious tuning of penalty parameters.
\end{abstract}

\section{Introduction}
In the recent decade, the idea of structured low-rank matrix approximation has stimulated a continually growing amount of work in dynamical system identification \cite{grossmann2009system, markovsky2013, liu2013nuclear, verhaegen2016n2sid, liu2021efficient, liu2023efficient}. By exploiting the low-rank property of data matrices germane to system order, one is capable of identifying low-order models directly from data. Its prevalence arises from the fact that, for model-based control design, it suffices to capture dominating process dynamics using a parsimonious structure \cite{seborg2010process}. Besides, the reduced complexity renders the model less susceptible to noise-corruption and alleviates over-fitting.


For the sake of computational tractability, a plenty of low-rank approximation methods resolve relaxed problems by means of the nuclear norm heuristic \cite{fazel2001rank,liu2010interior,recht2010guaranteed}. Although only convex programs need to be solved, these methods more or less suffer from performance loss due to relaxation; see \cite{mohan2010reweighted} for detailed discussions. The rank-constrained program (RCP) offers a more straightforward approach, which enforces strict satisfaction of rank constraints and thus helps to reduce approximation error. RCP has already found successful applications in various data-driven modeling tasks including continuous-time model identification \cite{liu2021efficient,liu2023efficient}, structured subspace identification \cite{yu2018identification,yu2019constrained},  factor analysis \cite{delgado2014rank}, to mention a few.

Because of the nonconvexity of rank constraint, RCP is NP-hard and notoriously difficult to solve. Current solution algorithms can be primarily classified into convex-concave procedure and its variants, which solve relaxed problems successively \cite{mohan2010reweighted,sun2017rank,yu2019constrained}, and alternating projection-based methods, which split the optimization problem into subproblems and solve them alternately. As a representative of the latter, the alternating direction method of multipliers (ADMM) has drawn immense attentions in resolving RCP due to its capability of starting from an infeasible solution and the flexibility in handling different problem structures \cite{boyd2011distributed,sun2017customized,wu2022maximum}. However, ADMM is known to suffer from slow convergence to a local optimum, especially in the final stage. Even though the convergence of ADMM has been proved for specific nonconvex problems, it either relies on strong assumptions that are nontrivial for RCP, or introduces additional approximations \cite{sun2017customized,wang2019global,zhang2019accelerating}. The crux is that, the performance of nonconvex ADMM heavily hinges on the penalty parameter, whose fine-tuning is cumbersome. As a rule-of-thumb, the penalty parameter shall be set small at the initial stage to escape from local optima, and then increased continually to enforce feasibility. However, such an updating strategy is susceptible to selection of hyper-parameters, which may lead to poor convergence (e.g. an imbalanced reduction of primal and dual residuals), or even divergent behaviors.

In this work, a novel self-adaptive penalty updating strategy for nonconvex ADMM is proposed to solve RCP problems in low-rank model identification. Motivated by convergence criterion of nonconvex ADMM, we propose to carry out first-order analysis of the increment of the augmented Lagrangian, which acts as an information carrier to effectively guide the automatic update of penalty parameter. In addition, by regarding ADMM as fixed-point iterations, the proposed self-adaptive strategy can be implemented within the Anderson acceleration scheme to further improve the convergence performance. Numerical experiments show that the proposed accelerated ADMM routine yields desirable convergent performance and manifests insensitivity towards initialization of penalty, thereby offering ease of implementation in face of real-world problems.

The layout of this paper is organized as follows. Section II revisits the RCP formulation of low-order system identification problems and the ADMM routine. Section III describes the new self-adaptive penalty updating strategy for nonconvex ADMM and its implementation in the Anderson acceleration scheme. Section IV showcases results of numerical examples, followed by final concluding remarks.

\textit{Notations:} The operators ${\rm rank}\{ \mathbf{A} \} $, $\mathbf{A}^\top$ and $|| \mathbf{A} ||_F$ denote the rank, the transpose and the Frobenius norm of a matrix $\mathbf{A}$, respectively. $\operatorname{vec}(\mathbf{A})$ denotes the vector formed by concatenating each column of matrix $\mathbf{A}$. $\mathbf{A}\circ\mathbf{B}$ denotes the element-wise product of two matrices of equal size. The identity matrix and the vector of all ones are denoted as $\mathbf{I}$ and $\mathbf{1}$, whose dimensions can be deemed from the context. The Hankel matrix constructed from a vector $\mathbf{x} \in \mathbb{R}^N$ with $n$ columns is defined as the following operator:
\begin{equation*}
\mathcal{H}_n(\mathbf{x}) = \begin{bmatrix}
x_1 & x_2 & \cdots & x_n \\
x_2 & x_3 & \cdots & x_{n+1} \\
\vdots & \vdots & \ddots & \vdots \\
x_{N-n+1} & x_{N-n+2} & \cdots & x_N
\end{bmatrix}.
\end{equation*}

\section{Rank-Constrained Identification of Low-Order Dynamical Models}
Consider the following single-input single-output discrete-time linear system:
\begin{equation}
\label{eq:dt_model}
y(t) = G(q)u(t)+v(t),
\end{equation}
where $t$ is the time index, $q$ is the forward shift operator, and $u(t)$, $y(t)$, $v(t)$ denote the input, output and additive noise at time $t$. The transfer function $G(q)$ could be expressed as $G(q)=\sum_{k=1}^\infty g_k q^{-1}$, where $\{ g_k \}_{k=1}^{\infty}$ denote the impulse response (IR). As the IR of a stable system decays to zero, truncation becomes a simple yet viable approach that gives rise to the finite impulse response (FIR) model $G(q)=\sum_{k=1}^l g_k q^{-1}$,
where $l$ is the length of FIR. Defining $\boldsymbol{\theta}=[g_1,\cdots,g_l]^\top\in\mathbb{R}^l$, model \eqref{eq:dt_model} can be approximated as:
\begin{equation}
    y(t)=\boldsymbol{\phi}(t)^\top\boldsymbol{\theta}+v(t),
    \label{eq:3}
\end{equation}
where $\boldsymbol{\phi}(t)=[u(t-1),\cdots,u(t-l)]^\top$. Given $N$ samples pairs $\{ \boldsymbol{\phi}(t), y(t) \}_{t=1}^N$, \eqref{eq:3} can be stacked in a vector form $\mathbf{y}=\mathbf{\Phi}\boldsymbol{\theta}+\mathbf{v}$, where $\mathbf{y}=[y(1),\cdots,y(N)]^\top$,
$\mathbf{\Phi}=[\boldsymbol{\phi}(1),\cdots,\boldsymbol{\phi}(N)]^\top$, and
$\mathbf{v}=[v(1),\cdots,v(N)]^\top$. The classical least-square estimate of $\boldsymbol{\theta}$ is then given by
\begin{equation}
\hat{\boldsymbol{\theta}}_\mathrm{LS}=\argmin_{\boldsymbol{\theta}} \|\mathbf{y}-\mathbf{\Phi}\boldsymbol{\theta}\|^2=(\bPhi^\top\bPhi)^{-1}\bPhi^\top\mathbf{y}.
\end{equation}
However, under heavy noise and inadequate sample size, least-square solution tends to be ill-condition. A remedy is given by the kernel-based methods, which control model complexity with the regularized least-square problem \cite{pillonetto2014kernel}:
\begin{equation}
\hat{\boldsymbol{\theta}}_\mathrm{ReLS}=\argmin_{\boldsymbol{\theta}} \|\mathbf{y}-\mathbf{\Phi}\boldsymbol{\theta}\|^2+\gamma \cdot \btheta^\top\mathbf{K}^{-1}\btheta,
\end{equation}
where $\gamma > 0$ is the tuning parameter and $\mathbf{K}$ is the kernel matrix constructed based on prior knowledge of impulse response, such as smoothness and exponential decay.

More recently, a new RCP perspective alternative to kernel-based methods has been cultivated in system identification. For a low-order model, its impulse response $\btheta$ typically embodies a low-rank structure, which is useful for controlling model complexity in a straightforward manner. More precisely, given a model with known order $r$, the Hankel matrix $\mathcal{H}_n(\btheta)$ has an exact rank of $r$ provided $n>r$, which motivates the following RCP formulation \cite{liu2021efficient,liu2023efficient}:

\begin{equation}
\label{eq:rcp0}
\begin{aligned}
\hat{\btheta}_\mathrm{RCP}=\argmin_{\btheta}\ & \left\|\mathbf{y}-\bPhi\btheta\right\|^2 \\
{\rm s.t. }\ & \operatorname{rank}\left\{\mathcal{H}_n(\btheta)\right\}=r
\end{aligned}
\end{equation}
By solving \eqref{eq:rcp0}, a low-order $G(q)$ can be exactly recovered from $\hat{\btheta}_\mathrm{RCP}$.\footnote{This property also holds for continuous-time transfer function $G(s)$.} For example, the RCP-based identification of first-order and second-order plus time delay (FOTD \& SOTD) models was considered in \cite{liu2021efficient} and its advantage in relay feedback controller tuning was further exploited in \cite{liu2023efficient}. The ADMM has been a popular approach to solving nonconvex RCP \cite{boyd2011distributed,sun2017customized}. By defining auxiliary variables $\mathbf{e}\in\mathbb{R}^N$ and $\mathbf{Z}\in\mathbb{R}^{(l+1-n)\times n}$, the RCP in \eqref{eq:rcp0} is reformulated as:
\begin{equation}
\label{eq:rcp}      
\begin{aligned}
\min_{\btheta,\mathbf{Z},\mathbf{e}}\ & \left\|\mathbf{e}\right\|^2\\
{\rm s.t. }\ & \mathbf{Z}+\mathcal{H}_n(\btheta)=\mathbf{0},~ \mathbf{e}+\bPhi\btheta-\mathbf{y}=\mathbf{0}\\
& \operatorname{rank}\left\{\mathbf{Z}\right\}=r\\
\end{aligned}
\end{equation}
The corresponding augmented Lagrangian ignoring the rank constraint is given by:
\begin{equation}
\label{eq:lagrange}
\begin{aligned}
&~ \mathcal{L}(\mathbf{Z},\mathbf{e},\btheta,\bLambda,\blambda;\beta) \\
=&~ \left\|\mathbf{e}\right\|^2-\blambda^{\top}\left(\mathbf{e}+\bPhi\btheta-\mathbf{y}\right)-\left\langle\bLambda, \mathbf{Z}+\mathcal{H}_n(\btheta)\right\rangle\\
&+\frac{\beta}{2}\left\|\mathbf{e}+\bPhi\btheta-\mathbf{y}\right\|^2+\frac{\beta}{2}\left\|\mathbf{Z}+\mathcal{H}_n(\btheta)\right\|_F^2,
\end{aligned}
\end{equation}
where $\{ \blambda\in\mathbb{R}^{N\times 1}, \bLambda\in\mathbb{R}^{(l+1-n)\times n} \}$ are dual variables and $\beta > 0$ is the penalty. Then the parameters are updated in a $\mathbf{Z}\to\mathbf{e}\to\boldsymbol{\theta}\to\bLambda\to\blambda$ order. Therein the updates of $\mathbf{e}$ and $\mathbf{Z}$ are given by:
\begin{gather}
\label{eq:e_update}
\mathbf{e}^{k+1}=\argmin_\mathbf{e}\left\|\mathbf{e}\right\|^2+\frac{\beta}{2}\left\|\mathbf{e}+\bPhi\btheta^k-\mathbf{y}-\frac{\blambda^k}{\beta}\right\|^2\\
\label{eq:Z_update}
\begin{aligned}
\mathbf{Z}^{k+1}=\argmin_\mathbf{Z}\ & \left\|\mathbf{Z}+\mathcal{H}_n(\btheta^k)-\frac{\bLambda^k}{\beta}\right\|_F^2\\
{\rm s.t. }\ & \operatorname{rank}\left\{\mathbf{Z}\right\}=r
\end{aligned}
\end{gather}
Problem \eqref{eq:e_update} admits a least-square solution. Without loss of generality, $l+1-n\geq n$ is assumed and thus the solution of \eqref{eq:Z_update} is given by the following truncated SVD:
\begin{equation}
\mathbf{Z}^{k+1}=\mathbf{U}^{k}\left(\tilde{\mathbf{I}}_r\circ\mathbf{\Sigma}^{k}\right)(\mathbf{V}^k)^\top,
\end{equation}
where matrices $\mathbf{U}\in\mathbb{R}^{(l+1-n)\times n}$, $\mathbf{\Sigma} = {\rm diag} \{ \sigma_1, \cdots, \sigma_n \} \in\mathbb{R}^{n\times n}$, and $\mathbf{V}\in\mathbb{R}^{n\times n}$ are derived from the SVD:
\begin{equation}
\label{eq:svd}
 -\mathcal{H}_n(\boldsymbol{\theta}^k)+{\frac{{\mathbf{\Lambda}}^k}{\beta^k}}  = \mathbf{U}^{k}\mathbf{\Sigma}^{k}(\mathbf{V}^k)^\top,
\end{equation}
and the diagonal matrix
\begin{equation}
\tilde{\mathbf{I}}_r=\mathrm{diag}\{\underbrace{1,\cdots,1}_{r},\underbrace{0,\cdots,0}_{n-r}\}\in\mathbb{R}^{n\times n}
\end{equation}
is used to ``pick out" $r$ largest singular values from $\mathbf{\Sigma}$ and helps preserve a low-rank structure. 

Different from the common two-block convex ADMM, which possesses appealing convergence property, the convergence performance of nonconvex ADMM heavily depends on the selection of penalty parameter $\beta$ \cite{wang2019global}. Among various applications of nonconvex ADMM, there are two primary options for deciding $\beta$ during iterations. One can trivially set $\beta$ to be a sufficiently large constant \cite{sun2017customized, liu2019linearized, wu2022maximum}, as implied by the established convergence results for nonconvex ADMM \cite{wang2019global}. However, an excessively large $\beta$ negatively impacts the convergence, especially the reduction of dual residuals. Another heuristic, which is referred to as \textit{multiplicative updating}, is to increase $\beta$ by a factor of $\rho > 1$ successively until a maximum $\beta_{\rm max}$ is reached or the residuals satisfy certain convergence condition \cite{chan2017plug, sun2021two, wang2021distributed}. Its rationality lies in that using a small $\beta$ at the initial stage helps to circumvent local optima. However, the values of hyper-parameters such as $\rho$ and $\beta_{\rm max}$ exert significant influence on the convergence and thus entail tedious fine-tuning.

\section{Accelerated Nonconvex ADMM For Rank-Constrained Model Identification}
In this section, a new penalty updating strategy for rank-constrained nonconvex ADMM is proposed by performing first-order analysis of the increment of augmented Lagrangian $\mathcal{L}$, which is then integrated with the Anderson acceleration scheme for improved efficiency. Before proceeding, we make some simplifications to facilitate subsequent analysis. In fact, the updates of $\mathbf{Z}$ and $\mathbf{e}$ are independent with each other and thus can be merged. Meanwhile, the updates of $\bLambda$ and $\blambda$ are also decoupled. Based on this, the following concatenated variables are introduced:
\begin{equation}
\mathbf{w}=\begin{bmatrix}
\text{vec}(\bZ)\\
\be
\end{bmatrix},\ 
\bmu=\begin{bmatrix}
\text{vec}(\bLambda)\\
\blambda
\end{bmatrix}.
\end{equation}
Then the RCP \eqref{eq:rcp} can be compactly expressed as:
\begin{equation}
\label{eq:rcp2}
\begin{aligned}
\min _{\mathbf{w},\btheta} &~ \|\mathbf{e}\|^2 \\
{\rm s.t. }&~ \mathbf{w}+\mathbf{Q}\btheta+\tilde{\mathbf{y}}=\mathbf{0}\\
&~ \operatorname{rank}\left\{\mathbf{Z}\right\}=r
\end{aligned}
\end{equation}
where
\begin{equation}
\mathbf{Q}=\begin{bmatrix}
\mathbf{M}\\
\bPhi
\end{bmatrix},\ \tilde{\mathbf{y}}=\begin{bmatrix}
\mathbf{0}\\
-\mathbf{y}
\end{bmatrix},
\end{equation}
and the coefficient matrix $\mathbf{M}$ is defined in order to fulfill $\mathbf{M}\mathbf{z}=\operatorname{vec}(\mathcal{H}_n(\mathbf{z}))$. Then the augmented Lagrangian $\mathcal{L}$ could be rewritten as:
\begin{equation}
\label{eq:lb}
\begin{aligned}
&\mathcal{L}(\mathbf{w},\btheta,\bmu;\beta) \\
=&~ \|\mathbf{e}\|^2-\bmu^\top\left(\mathbf{w}+\mathbf{Q}\btheta+\tilde{\mathbf{y}}\right) +\frac{\beta}{2}\left\|\mathbf{w}+\mathbf{Q}\btheta+\tilde{\mathbf{y}}\right\|^2,
\end{aligned}
\end{equation}
Consequently, the ADMM iterations can be summarized as follows:
\begin{gather}
\label{eq:p_update}
\begin{aligned}
\mathbf{w}^{k+1}=\argmin_\mathbf{w}&~\mathcal{L}(\mathbf{w},\btheta^k,\bmu^k;\beta^k)\\
{\rm s.t.} &~ \operatorname{rank}\left\{\mathbf{Z}\right\}=r
\end{aligned}\\
\label{eq:q_update}
\btheta^{k+1}=\argmin_{\btheta}\mathcal{L}(\mathbf{w}^{k+1},\btheta,\bmu^k;\beta^k)\\
\label{eq:w_update}
\bmu^{k+1}=\bmu^k-\beta^k\left(\mathbf{w}^{k+1}+\mathbf{Q}\btheta^{k+1}+\tilde{\mathbf{y}}\right)
\end{gather}

\subsection{Self-adaptive penalty in nonconvex ADMM}
The proposed updating strategy of penalty parameter is inspired from the descent property of $\mathcal{L}$, which is pivotal for establishing the convergence of nonconvex ADMM towards a stationary point \cite[Property 2]{wang2019global}. The increment of $\mathcal{L}$ in the $k$-th iteration reads as:
\begin{equation}
\label{eq:Delta_L}
\Delta\mathcal{L}^k \triangleq \mathcal{L}(\mathbf{w}^{k+1},\btheta^{k+1},\bmu^{k+1};\beta^k)-\mathcal{L}(\mathbf{w}^k,\btheta^k,\bmu^k;\beta^k).
\end{equation}
Given $\mathbf{w}^k$, $\btheta^k$, $\bmu^k$ and $\beta^k$, new updates $\mathbf{w}^{k+1}$, $\btheta^{k+1}$, and $\bmu^{k+1}$ can be uniquely determined. Thus, $\Delta\mathcal{L}^k$ can be thought of a function of $\{ \mathbf{w}^k,\btheta^k,\bmu^k,\beta^k\}$. Notably, its partial derivative $\partial\Delta\mathcal{L}^k / \partial{\beta^k}$ contains useful information about the impact of varying $\beta^k$ on $\Delta\mathcal{L}^k$. When $\partial\Delta\mathcal{L}^k / \partial{\beta^k}<0$, it indicates that a larger $\beta^k$ is expected to make the augmented Lagrangian possess a better descending characteristic. When $\partial\Delta\mathcal{L}^k / \partial{\beta^k}>0$, a smaller $\beta^k$ is suggested. This motivates the following heuristic for updating $\beta$ in nonconvex ADMM:
\begin{equation}
\label{eq:beta_update}
\beta^{k+1}= \left \{ \begin{aligned}
& \beta^k\cdot\rho_\text{inc},~~ \frac{\partial \Delta\mathcal{L}^k}{\partial{\beta^k}}<0, \\
& \beta^k/\rho_\text{dec},~~ \frac{\partial \Delta\mathcal{L}^k}{\partial{\beta^k}}>0,\\
& \beta^k,~~~~~~~~ \frac{\partial \Delta\mathcal{L}^k}{\partial{\beta^k}}=0,
\end{aligned} \right .
\end{equation}
where $\rho_\text{inc}>\rho_\text{dec}>1$ are, respectively, the increase and decrease factors of $\beta$.\footnote{The reason for setting $\rho_\mathrm{inc}>\rho_\mathrm{dec}$ lies in that, a small $\beta$ may lead to divergence, so one needs to be more cautious in decreasing the value of $\beta$.}

\begin{rem}
\label{rem1}
The above heuristic is reminiscent of the well-known residual-based adaptive penalty updating policy for convex problems, which is based on inspecting the relation between primal and dual residuals \cite{he2000alternating, boyd2011distributed}:
\begin{equation}
\label{eq:22}
\beta^{k+1}= \left \{ \begin{aligned}
&\beta^{k}\cdot\rho_\mathrm{inc},~~\left\|\boldsymbol{\varepsilon}^{k}_\mathrm{p}\right\|^2>\kappa\left\|\boldsymbol{\varepsilon}^{k}_\mathrm{d}\right\|^2, \\
&\beta^{k}/\rho_\mathrm{dec},~~\left\|\boldsymbol{\varepsilon}^{k}_\mathrm{d}\right\|^2>\kappa\left\|\boldsymbol{\varepsilon}^{k}_\mathrm{p}\right\|^2, \\
&\beta^k,~~~~~~~~~\text {otherwise},
\end{aligned}\right.
\end{equation}
where primal and dual residuals $\boldsymbol{\varepsilon}^k_\mathrm{p}$ and $\boldsymbol{\varepsilon}^k_\mathrm{d}$ are defined by
\begin{gather}
\boldsymbol{\varepsilon}_\mathrm{p}^{k+1}=\mathbf{w}^{k+1}+\mathbf{Q}\boldsymbol{\theta}^{k+1}+\tilde{\mathbf{y}},\\
\boldsymbol{\varepsilon}_\mathrm{d}^{k+1}=\beta^k\mathbf{Q}(\btheta^{k+1}-\btheta^k),
\end{gather}
and $\kappa>1$ is a threshold parameter. Notwithstanding its effectiveness on a variety of convex problems, the updating rule \eqref{eq:22} still has some pitfalls. In essence, such a strategy takes $\left\|\boldsymbol{\varepsilon}^{k}_\mathrm{p}\right\|^2+\left\|\boldsymbol{\varepsilon}^{k}_\mathrm{d}\right\|^2$ as an indicator of convergence and the update of $\beta$ is oriented towards equal primal and dual residuals. However, recent studies have pointed out that there is no theoretical guarantee that identical primal and dual residuals are optimal for the convergence \cite{wohlberg2017admm}, which will also be evidenced in case study.
\end{rem}

The key challenge of executing the self-adaptive strategy consists in the calculation of $\partial\Delta\mathcal{L}^k / \partial{\beta^k}$. Delving into ADMM iterations \eqref{eq:p_update}-\eqref{eq:w_update}, we have the following supporting result.
\begin{thm}
\label{lemma1}
The following relations hold during iterations of ADMM:
\begin{gather}
\label{eq:Mw=0}
\mathbf{Q}^\top\bmu^{k+1}=\mathbf{0},\\
\label{eq:q_p}
\btheta^{k+1}=-(\mathbf{Q}^\top\mathbf{Q})^{-1}\mathbf{Q}^\top\left(\mathbf{w}^{k+1}+\tilde{\mathbf{y}}\right).
\end{gather}
\end{thm}

\begin{proof}
In \eqref{eq:q_update}, the update of $\btheta$ admits the following closed-form solution:
\begin{equation}
\label{eq:q_update_solution}
\btheta^{k+1}=-(\mathbf{Q}^\top\mathbf{Q})^{-1}\mathbf{Q}^\top\left(\mathbf{w}^{k+1}+\tilde{\mathbf{y}}-\frac{\bmu^k}{\beta^k}\right),~\forall k\geq 0.
\end{equation}
Plugging it into \eqref{eq:w_update} and left-multiplying $\mathbf{Q}^\top$ on both sides, we arrive at \eqref{eq:Mw=0}. Substituting \eqref{eq:Mw=0} into \eqref{eq:q_update_solution} then yields \eqref{eq:q_p}.
\end{proof}

Now we are in the position to describe the derivation of $\partial\Delta\mathcal{L}^k / \partial{\beta^k}$. Using Theorem \ref{lemma1}, the increment of $\mathcal{L}$ can be rewritten as:
\begin{equation}
\begin{aligned}
&\Delta\mathcal{L}^k=\|\mathbf{e}^{k+1}\|^2-\|\mathbf{e}^k\|^2\\
&-\left[\bmu^k-\beta^k\left(\mathbf{w}^k+\mathbf{Q}\btheta^k+\tilde{\mathbf{y}}\right)\right]^\top\left(\mathbf{w}^{k+1}-\mathbf{w}^k\right)\\
&+\beta^k\left[\left\|\mathbf{P}(\mathbf{w}^{k+1}-\mathbf{w}^k)\right\|^2/2+\left\|\mathbf{P}(\mathbf{w}^{k+1}+\tilde{\mathbf{y}})\right\|^2\right],
\end{aligned}
\end{equation}
where $\mathbf{P}=\mathbf{I}-\mathbf{Q}\left(\mathbf{Q}^\top\mathbf{Q}\right)^{-1}\mathbf{Q}^\top$. In this way, $\btheta^{k+1}$ and $\bmu^{k+1}$ can be eliminated. One then obtains:
\begin{equation}
\begin{aligned}
&\frac{\partial\Delta\mathcal{L}^k}{\partial{\beta^k}}=
\left (\frac{\partial\mathbf{w}^{k+1}}{\partial{\beta^k}} \right )^\top{\left[\begin{bmatrix}
2\mathbf{e}^{k+1}\\
\mathbf{0}
\end{bmatrix}-\bmu^k\right.}\\
&\left.+\beta^k\left(\mathbf{w}^k+\mathbf{Q}\btheta^k+\tilde{\mathbf{y}}+\mathbf{P}^\top\mathbf{P}(3\mathbf{w}^{k+1}-\mathbf{w}^k+2\tilde{\mathbf{y}})\right)\right]\\
&+(\mathbf{w}^k+\mathbf{Q}\boldsymbol{\theta}^k+\tilde{\mathbf{y}})^\top(\mathbf{w}^{k+1}-\mathbf{w}^k)\\
&+\left\|\mathbf{P}(\mathbf{w}^{k+1}-\mathbf{w}^k)\right\|^2/2+\left\|\mathbf{P}(\mathbf{w}^{k+1}+\tilde{\mathbf{y}})\right\|^2.
\end{aligned}
\end{equation}
Clearly, only $\partial\mathbf{w}^{k+1} / \partial{\beta^k}$ needs to be computed, which is defined as:
\begin{equation}
\frac{\partial \mathbf{w}^{k+1}}{\partial{\beta^k}}=\begin{bmatrix}
\mathrm{vec}\left( \frac{\partial \mathbf{Z}^{k+1}}{\partial{\beta^k}}\right)\\
\frac{\partial \mathbf{e}^{k+1}}{\partial{\beta^k}}
\end{bmatrix}.
\end{equation}
We first deal with the computation of $\partial \mathbf{Z}^{k+1} / \partial{\beta^k}$. Because $\mathbf{Z}^{k+1}=\mathbf{U}^{k}(\tilde{\mathbf{I}}_r\circ\mathbf{\Sigma}^{k})(\mathbf{V}^k)^\top$, its derivative can be derived as
\begin{equation}
\begin{aligned}
\frac{\partial \mathbf{Z}^{k+1}}{\partial{\beta^k}}=&\frac{\partial \mathbf{U}^{k}}{\partial{\beta^k}}\left(\tilde{\mathbf{I}}_r\circ\mathbf{\Sigma}^k\right)(\mathbf{V}^k)^\top+\mathbf{U}^k\left(\tilde{\mathbf{I}}_r\circ \frac{\partial \mathbf{\Sigma}^k}{\partial{\beta^k}}\right)(\mathbf{V}^k)^\top\\
&+\mathbf{U}^k\left(\tilde{\mathbf{I}}_r\circ\mathbf{\Sigma}^k\right)\frac{\partial (\mathbf{V}^k)^\top}{\partial{\beta^k}},
\end{aligned}
\end{equation}
which involves the derivatives of $\mathbf{U}^k$, $\mathbf{\Sigma}^k$ and $\mathbf{V}^k$. Recalling that $\mathbf{U}^k\mathbf{\Sigma}^k(\mathbf{V}^k)^\top=-\mathcal{H}_n(\mathbf{z}^k)+{\mathbf{\Lambda}^k}/{{\beta^k}}$, we can attain these derivatives using matrix calculus of SVD \cite{bodewig2014matrix}:
\begin{equation}
\begin{split}
\frac{\partial{\mathbf{\Sigma}^k}}{\partial{\beta^k}}=&~\mathbf{I}\circ\boldsymbol{\Theta}^k\\
\frac{\partial{\mathbf{U}^k}}{\partial{\beta^k}}=&~\mathbf{U}^k\left(\mathbf{G} \circ\left[\boldsymbol{\Theta}^k\mathbf{\Sigma}^k+\mathbf{\Sigma}^k (\boldsymbol{\Theta}^k)^\top\right]\right)\\
&+\left(\mathbf{I}-\mathbf{U}^k (\mathbf{U}^k)^{\top}\right) \frac{\partial\mathbf{\Omega}^k}{\partial\beta^k} \mathbf{V}^k (\mathbf{\Sigma}^k)^{-1} \\
\frac{\partial\mathbf{V}^k}{\partial\beta^k}=&~\mathbf{V}^k\left(\mathbf{G} \circ\left[\mathbf{\Sigma}^k \boldsymbol{\Theta}^k+(\boldsymbol{\Theta}^k)^\top\mathbf{\Sigma}^k\right]\right)\\
&+\left(\mathbf{I}-\mathbf{V}^k(\mathbf{V}^k)^{\top}\right) \frac{\partial(\mathbf{\Omega}^k)^{\top}}{\partial\beta^k}\mathbf{U}^k (\mathbf{\Sigma}^k) ^ { - 1 }
\end{split}
\end{equation}
where 
\begin{equation}
\boldsymbol{\Theta}^k=(\mathbf{U}^k)^\top\frac{\partial\mathbf{\Omega}^k}{\partial\beta^k}\mathbf{V}^k
=-\frac{1}{\beta_k^2}(\mathbf{U}^k)^\top\mathbf{\Lambda}^k\mathbf{V}^k,
\end{equation}
and entries of $\mathbf{G}\in\mathbb{R}^{n\times n}$ are given by
\begin{equation}
G_{ij}=\left \{ 
\begin{aligned}
& 1 / (\sigma_j^2-\sigma_i^2),~ i \neq j \\
& 0,~~~~~~~~~~~~~~ i=j
\end{aligned} \right .
\end{equation}

We then proceed to compute $\partial \mathbf{e}^{k+1} / \partial{\beta^k}$. According to the solution to problem \eqref{eq:e_update}, the derivative can be easily derived as
\begin{equation}
\frac{\partial \mathbf{e}^{k+1}}{\partial{\beta^k}}=\frac{1}{\beta^k+2}{(\mathbf{y} - \mathbf{e}^{k+1} - \bPhi\btheta^k)}.
\end{equation}

\begin{rem}
Since SVD used in computing $\partial \mathbf{Z}^{k+1} / \partial{\beta^k}$ has been obtained in the previous iteration \eqref{eq:svd} and the computation of $\partial \mathbf{e}^{k+1} / \partial{\beta^k}$ has a simple form, the computation of $\partial\Delta\mathcal{L}^k / \partial{\beta^k}$ only calls for minor extra computational cost as compared to the original ADMM procedure.
\end{rem}

\subsection{Implementation with Anderson acceleration}
Next, we describe the implementation of the self-adaptive penalty within the Anderson acceleration scheme. It follows from \eqref{eq:p_update}-\eqref{eq:w_update} that, $\btheta^{k+1}$ and $\bmu^{k+1}$ depend solely on $\btheta^k$ and $\bmu^k$ and are independent of $\mathbf{w}^k$. Thus, letting $\boldsymbol{\xi} = [\btheta^\top ~ \bmu^\top]^\top$, the above ADMM iteration can interpreted as the fixed-point iteration $\bxi^{k+1}=\mathcal{G}(\bxi^k)$, to which the Anderson acceleration can be applied \cite{zhang2019accelerating}. By keeping track of results in $m$ past iterations, the accelerated iteration for $\bxi^{k+1}$ is given by:
\begin{equation}
\label{eq:aa}
\bxi_\mathrm{AA}^{k+1} = \mathcal{G}(\bxi^{k}) - \sum_{j=1}^m{\alpha_j^{k}}^*\left[\mathcal{G}(\bxi^{k-j+1}) - \mathcal{G}(\bxi^{k-j})\right].
\end{equation}
Coefficients $({\alpha_1^{k}}^*,{\alpha_2^{k}}^*,\ldots,{\alpha_m^{k}}^*)$ are the solution to the following least-squares problem:
\begin{equation}
\label{eq:eta_least_squares}
\mathop{\min}_{(\alpha_1^k,\alpha_2^k,\ldots,\alpha_m^k)} \left\|{\boldsymbol{\eta}^{k} - \sum_{j=1}^m\alpha_j^k\left(\boldsymbol{\eta}^{k-j+1} - \boldsymbol{\eta}^{k-j}\right)}\right\|^2.
\end{equation}
where $\boldsymbol{\eta}^k=\mathcal{G}(\bxi^k)-\bxi^k$ denotes the residual in fixed-point iterations.

In addition, a fall-back strategy is necessary for the steady execution of Anderson acceleration. Combined residual is a commonly used criterion for convergence performance in ADMM, which is defined as:
\begin{equation}
\label{eq:combined_residual}
\varepsilon^{k+1}=\beta^k\|\boldsymbol{\varepsilon}_\mathrm{p}^{k+1}\|^2+\|\boldsymbol{\varepsilon}_\mathrm{d}^{k+1}\|^2/{\beta^k}.
\end{equation}

The entire ADMM procedure with Anderson acceleration is summarized in Algorithm 1. In each iteration, Anderson acceleration is attempted by utilizing $m>0$ past results, where $m$ is bounded by $m_{\rm max}$ to avoid negative impact of distant iterations \cite{zhang2019accelerating}. The efficacy of acceleration is then evaluated by the combined residual $\varepsilon$. If a smaller $\varepsilon$ is achieved, the accelerated outcome is then accepted. Otherwise, the algorithm backtracks to the last generic ADMM iteration.

\begin{rem}
The success of Anderson acceleration entails the convergence of ADMM itself \cite{zhang2019accelerating}. Since the self-adaptive updating strategy improves the convergence performance of ADMM, it may also improve the efficacy of Anderson acceleration.
\end{rem}

\begin{algorithm}[htbp]
\caption{Nonconvex ADMM for Rank-Constrained Model Identification with Self-Adaptive Penalty and Anderson Acceleration}
\begin{tabular}{ll}
\textbf{Input:}
& Initial $\mathbf{w}^0$, $\btheta^0$ and $\bmu^0$. Initial penalty $\beta^0$. \\
& Increase and decrease factors $\rho_\mathrm{inc}>\rho_\mathrm{dec}>1$.\\
& The maximum number of past iterations used\\ 
& for acceleration $m_{\rm max}\geq1$.\\
& Maximum iteration number $k_\mathrm{max}$.\\
& Stopping tolerance $\varepsilon_\mathrm{tol}$.
\end{tabular}
\begin{algorithmic}[1]
\State Set iteration count $k=0$.
\While {$\mathrm{TRUE}$}
\State // Run one iteration of ADMM
\State Update $\mathbf{w}^{k+1}$, $\btheta^{k+1}$, $\bmu^{k+1}$ using \eqref{eq:p_update}-\eqref{eq:w_update}.
\State // Compute the combined residual
\State Compute $\varepsilon$ using \eqref{eq:combined_residual}.
\If {$\mathrm{reset}==\mathrm{TRUE} \textbf{ or } \varepsilon < \varepsilon_\mathrm{prev}$} 
\State // Record the latest accepted iterate
\State $\bxi_\mathrm{rec}=\bxi^{k+1}.\  \varepsilon_\mathrm{prev} = \varepsilon.\ \mathrm{reset} = \mathrm{FALSE}.$
\State // Compute the accelerated iterate
\State $m=\min(m_{\rm max},k).$
\State Compute $\bxi_{AA}^{k+1}$ using \eqref{eq:aa}, \eqref{eq:eta_least_squares}.
\State Let $\bxi^{k+1}=\bxi_{AA}^{k+1}$.
\State Update $\beta^{k+1}$ using \eqref{eq:beta_update}.
\State $k=k+1.$
\Else \ // Backtrack to the last accepted record
\State Let $\bxi^{k}=\bxi_\mathrm{rec}.\ \mathrm{reset}=\mathrm{TRUE}.$
\EndIf
\If {$k>k_\mathrm{max} \textbf{ or } \varepsilon<\varepsilon_{\rm tol}$}
\State \Return $\btheta_\mathrm{rec}.$
\EndIf
\EndWhile
\end{algorithmic}
\end{algorithm}

\section{Numerical Examples}
We carry out a series of simulation experiments to illustrate the advantages of the proposed method, where the task is to identify the following continuous-time SOTD model:
\begin{equation}
G(s)=\frac{0.2s+1}{1.5s^2+0.6s+1}e^{-3s}.
\end{equation}
A relay feedback control lasting $50$s is imposed on the process to excite the system and generate input and output data for identification. The input and output data are sampled at an interval of $0.5$s, and Gaussian white noise with variance of $\sigma_N^2=0.01$ is added to the output. In each simulation, 100 data points are collected in each simulation, and then 100 Monte Carlo simulations are carried out in total based on which the average performance can be evaluated. The ADMM routines under different updating strategies, including constant penalty, the multiplicative updating and the proposed self-adaptive updating with/without Anderson acceleration, are executed to solve the RCP problem \eqref{eq:rcp0} and identify the model. The value of $\boldsymbol{\theta}$ is initialized using the kernel-based method in \cite{liu2023efficient}. 

\subsection{Constant $\beta$}
First, we investigate the performance of ADMM under constant $\beta$. We consider four different $\beta$ ($\beta=1,10,100,1000$). In each case, the average squared norms of primal residual $\boldsymbol{\varepsilon}_\mathrm{p}$ and dual residual $\boldsymbol{\varepsilon}_\mathrm{d}$ are recorded to evaluate the performance, as profiled in Fig. \ref{fig:constant}. It can be seen that when $\beta\leq 10$, ADMM shows poor performance due to the inadequate penalty, and with $\beta\geq 100$, a large dual residual is encountered. Henceforth, using a constant penalty can hardly achieve a balanced reduction between primal and dual residuals.
\begin{figure}[htbp]
\centering
\includegraphics[width=\linewidth]{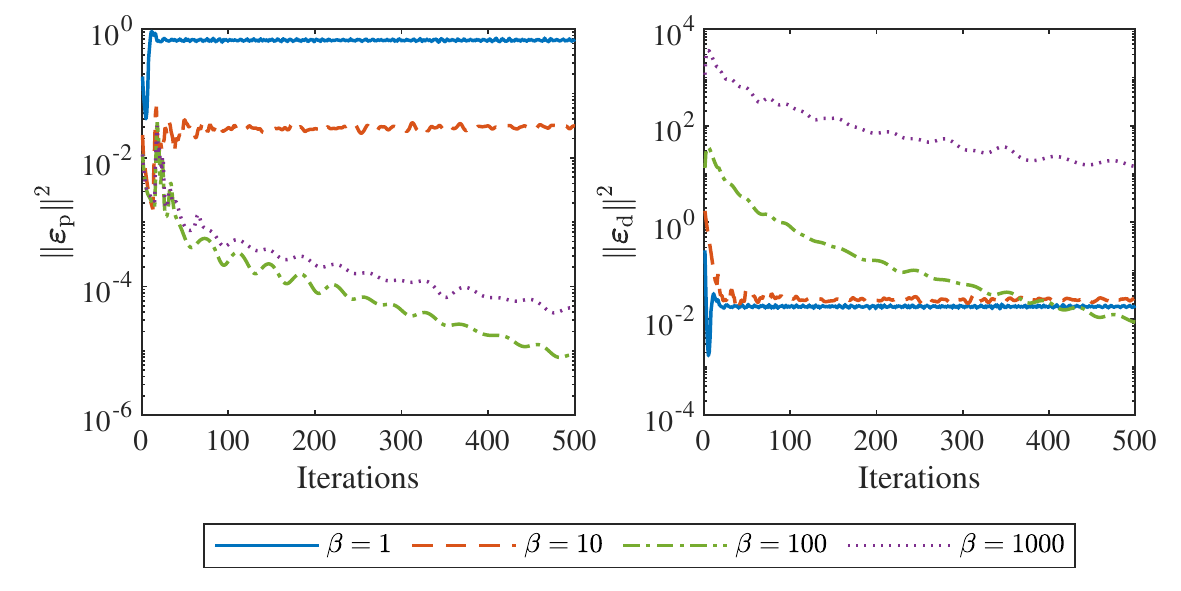}
\caption{The average squared norm of primal residual $\boldsymbol{\varepsilon}_\mathrm{p}$ and dual residual $\boldsymbol{\varepsilon}_\mathrm{d}$ in constant penalty in 100 Monte Carlo runs.}
\label{fig:constant}
\end{figure}

\subsection{Multiplicative updating strategy}
Then we investigate the performance of ADMM under multiplicative updating. We conduct two experiments. In the former one, the initial value of beta is uniformly set to $\beta^0=1$, and we consider different update strategies under different setups ($\rho=1.01,1.1$ and $\beta_\mathrm{max}=10,100,1000$). In the latter one, the update strategy is fixed as $\rho=1.01,\beta_\mathrm{max}=100$, and different initial values of $\beta$ is considered $(\beta^0=0.1,1,10,100)$. The average squared norms of primal residual and dual residual are illustrated in Fig. \ref{fig2}.\footnote{Note that in Fig. \ref{fig:multi1}, the curve under $\rho=1.01,\ \beta_\mathrm{max}=1000$ is dismissed because $1.01^{500}<1000$.} In Fig. \ref{fig:multi1}, the ADMM under multiplicative updating strategy exhibits a slow convergence with small $\rho$ and $\beta_\mathrm{max}$, e.g., $\rho=1.01,\beta_\mathrm{max}=10$, and leads to large dual residual with large ones. In other cases, even though the primal and dual residuals continually decrease in the early stage, the convergence gets bogged down after $\beta$ reaches the maximum. Fig. \ref{fig:multi2} once more shows that small $\beta^0$ gives rise to slow convergence or even divergence, while large $\beta^0$ yields a large dual residual. To sum up, the practical performance of ADMM is sensitive to the choice of hyper-parameters and it is difficult to concurrently and steadily suppress two residuals even with careful tuning.
\begin{figure}[htbp]
\centering
\subfigure[Performance under different $\rho$ and $\beta_\mathrm{max}$.]{\includegraphics[width=\linewidth]{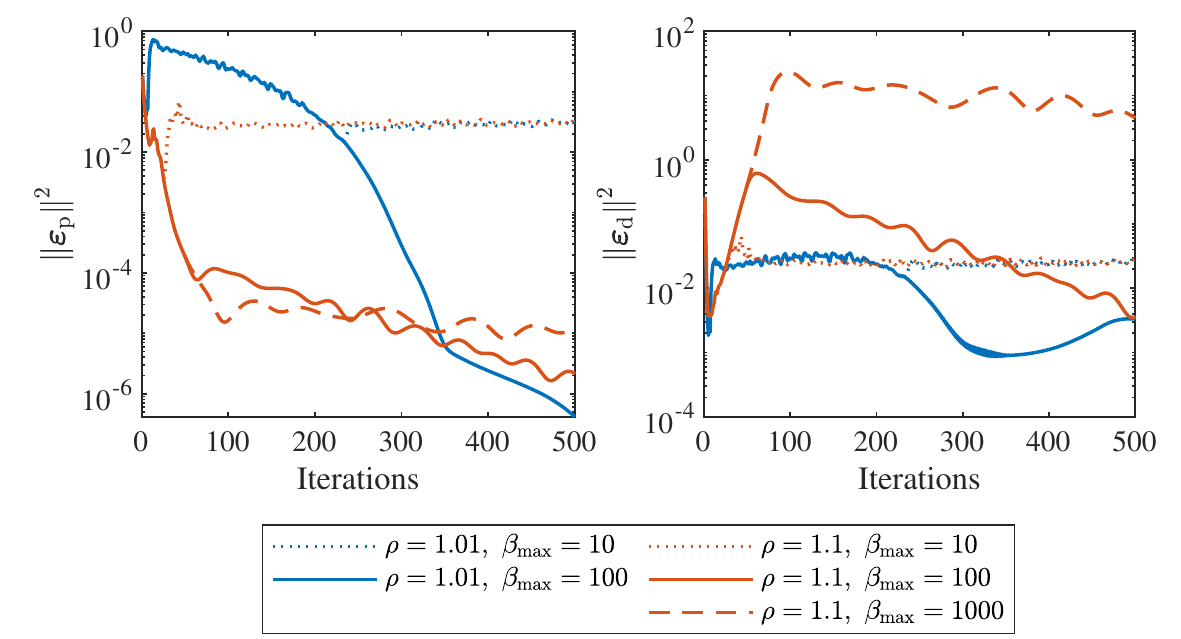}\label{fig:multi1}}\hfill
\subfigure[Performance under different $\beta_0$.]{\includegraphics[width=\linewidth]{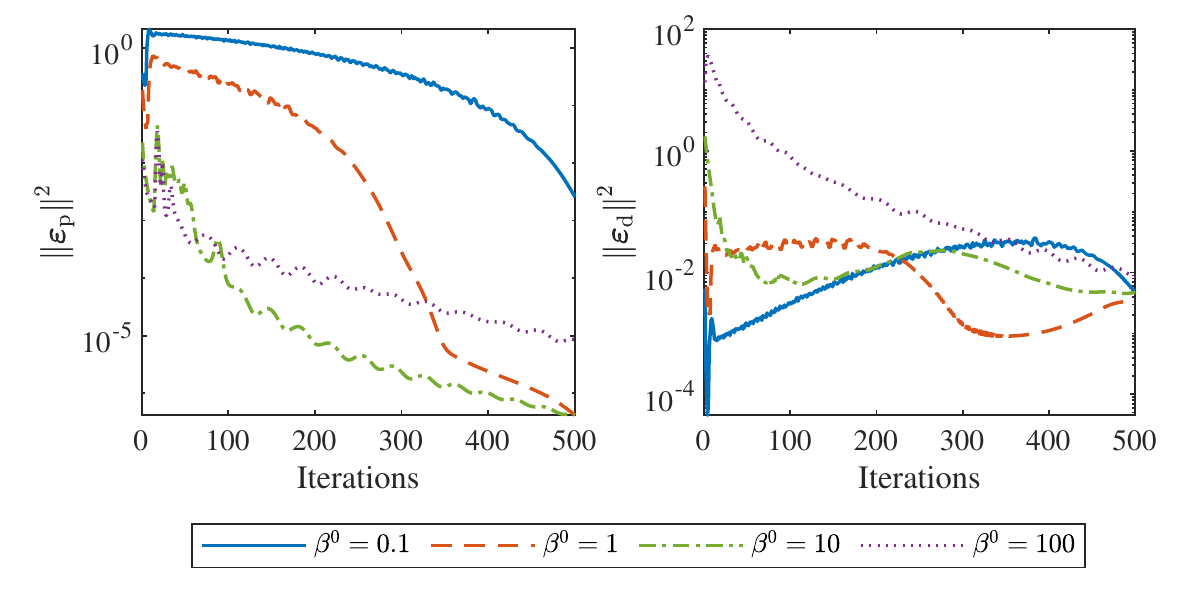}\label{fig:multi2}}
\caption{The average squared norm of primal residual $\boldsymbol{\varepsilon}_\mathrm{p}$ and dual residual $\boldsymbol{\varepsilon}_\mathrm{d}$ in multiplicative updating strategy in 100 Monte Carlo runs.}
\label{fig2}
\end{figure}

\subsection{Residual-based adaptive updating strategy}
Next, before delving into the proposed self-adaptive updating method, we investigate the performance of ADMM under the existing residual-based adaptive updating strategy mentioned in Remark \ref{rem1}, which endeavors to obtain equal primal residual and dual residual during iterations \cite{he2000alternating, boyd2011distributed}. Four different initial values of $\beta$ are considered ($\beta^0=0.1,1,10,100$). Other hyper-parameters in \eqref{eq:22} are set as $\kappa=10,\ \rho_\text{inc}=\rho_\text{dec}=1.02$. The average squared norms of primal residual and dual residual are profiled in Fig. \ref{fig:norm_adap}. It can be clearly seen that, under all $\beta^0$, the residual-based adaptive update strategy fails to converge, thereby revealing its incapability of handling nonconvex ADMM.
\begin{figure}[htbp]
\centering
\includegraphics[width=\linewidth]{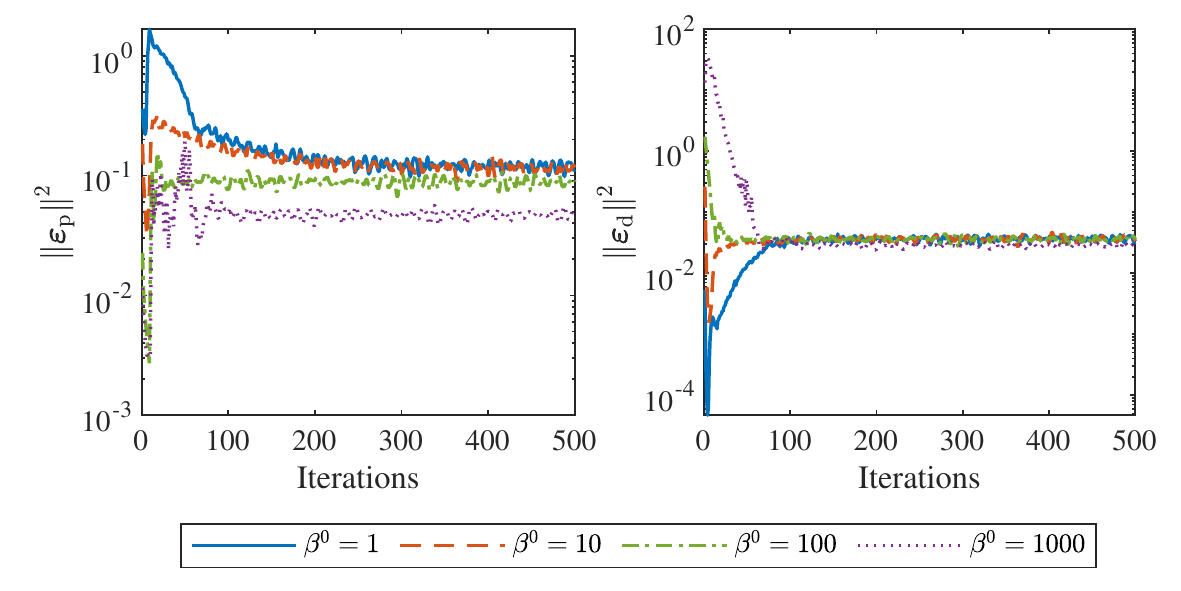}
\caption{The average squared norm of primal residual $\boldsymbol{\varepsilon}_\mathrm{p}$ and dual residual $\boldsymbol{\varepsilon}_\mathrm{d}$ in residual-based adaptive updating strategy in 100 Monte Carlo runs.}
\label{fig:norm_adap}
\end{figure}

\subsection{The proposed self-adaptive updating strategy}
Finally, we testify the performance of ADMM under the proposed self-adaptive updating strategy, which is compared with that under multiplicative updating strategy. For the proposed method, hyper-parameters are specified as $\rho_\mathrm{inc}=1.05$, $\rho_\mathrm{dec}=1.02$. For multiplicative updating strategy, the fine-tuned setup with $\rho=1.01$ and $\beta_\mathrm{max}=100$ is adopted. For both methods, we consider cases with and without Anderson acceleration, and $m$ is limited by $m_\mathrm{max}=5$. In all cases the initial penalty is set to $\beta^0=1$. The evolution of residuals is displayed in Fig. \ref{fig:compare}. It can be seen that self-adaptive penalty yields much better performance than multiplicative updating. On one hand, the self-adaptive updating results in significantly faster convergence in the early stage as it allows for using a large $\rho_\mathrm{inc}$. On the other hand, thanks to the adaptability of $\beta^k$, both primal and dual residuals can be stably reduced. Moreover, further improvement can be made by performing Anderson acceleration, which yields much smaller residuals under both strategies. When applied to the proposed method, it can also be observed that oscillation is effectively damped with Anderson acceleration.

\begin{figure}[htbp]
\centering
\includegraphics[width=\linewidth]{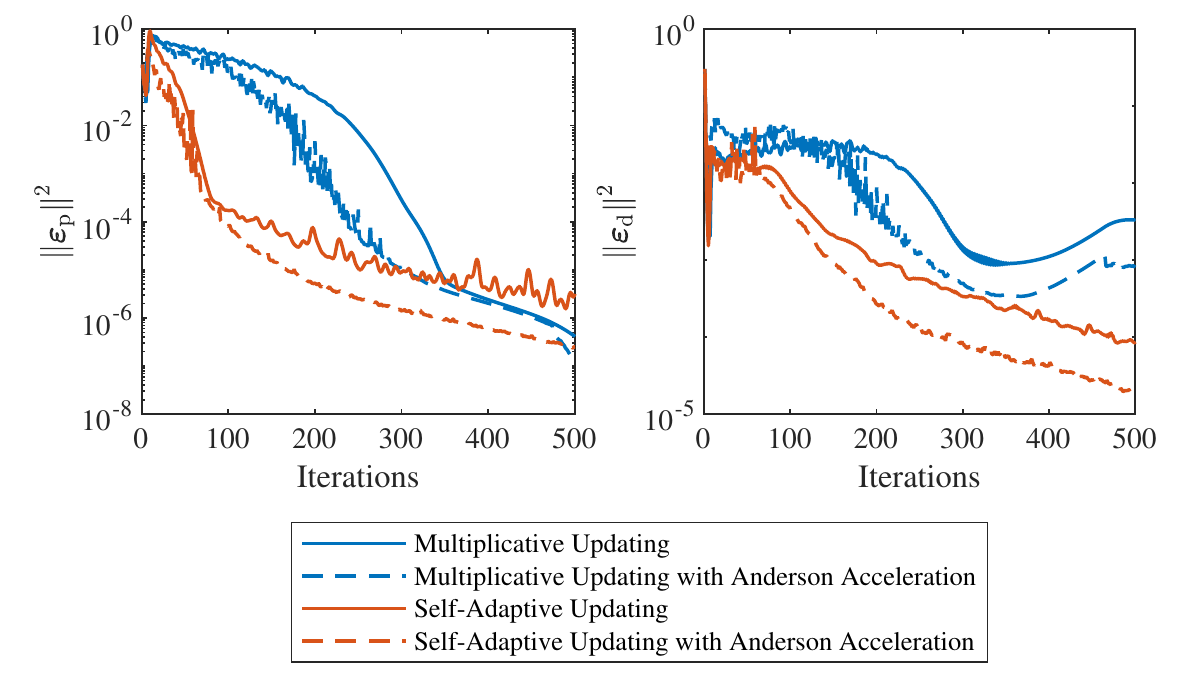}
\caption{The average squared norm of primal residual $\boldsymbol{\varepsilon}_\mathrm{p}$ and dual residual $\boldsymbol{\varepsilon}_\mathrm{d}$ in multiplicative updating and proposed self-adaptive updating strategies in 100 Monte Carlo runs.}
\label{fig:compare}
\end{figure}

We finally highlight that, as a merit of the proposed method, its adaptability enables to effectively accommodate different initial values of $\beta$. We testify the performance of ADMM under the proposed updating strategy and Anderson acceleration with $\beta^0=0.1,1,10,100$. Apart from $||\boldsymbol{\varepsilon}_\mathrm{p}||^2$ and $||\boldsymbol{\varepsilon}_\mathrm{d}||^2$, the evolution of $\beta$ is also recorded, as depicted in Fig. \ref{fig:adap}, where both primal and dual residuals can be effectively reduced and a steady convergence is attained. Most importantly, despite different initialization, the values of $\beta$ get close after iterations and eventually converge. This demonstrates that the proposed self-adaptive strategy is somewhat immune to initial values of $\beta$ and is thus advantageous over generic updating schemes in practical usage.

\begin{figure}[htbp]
\centering
\includegraphics[width=\linewidth]{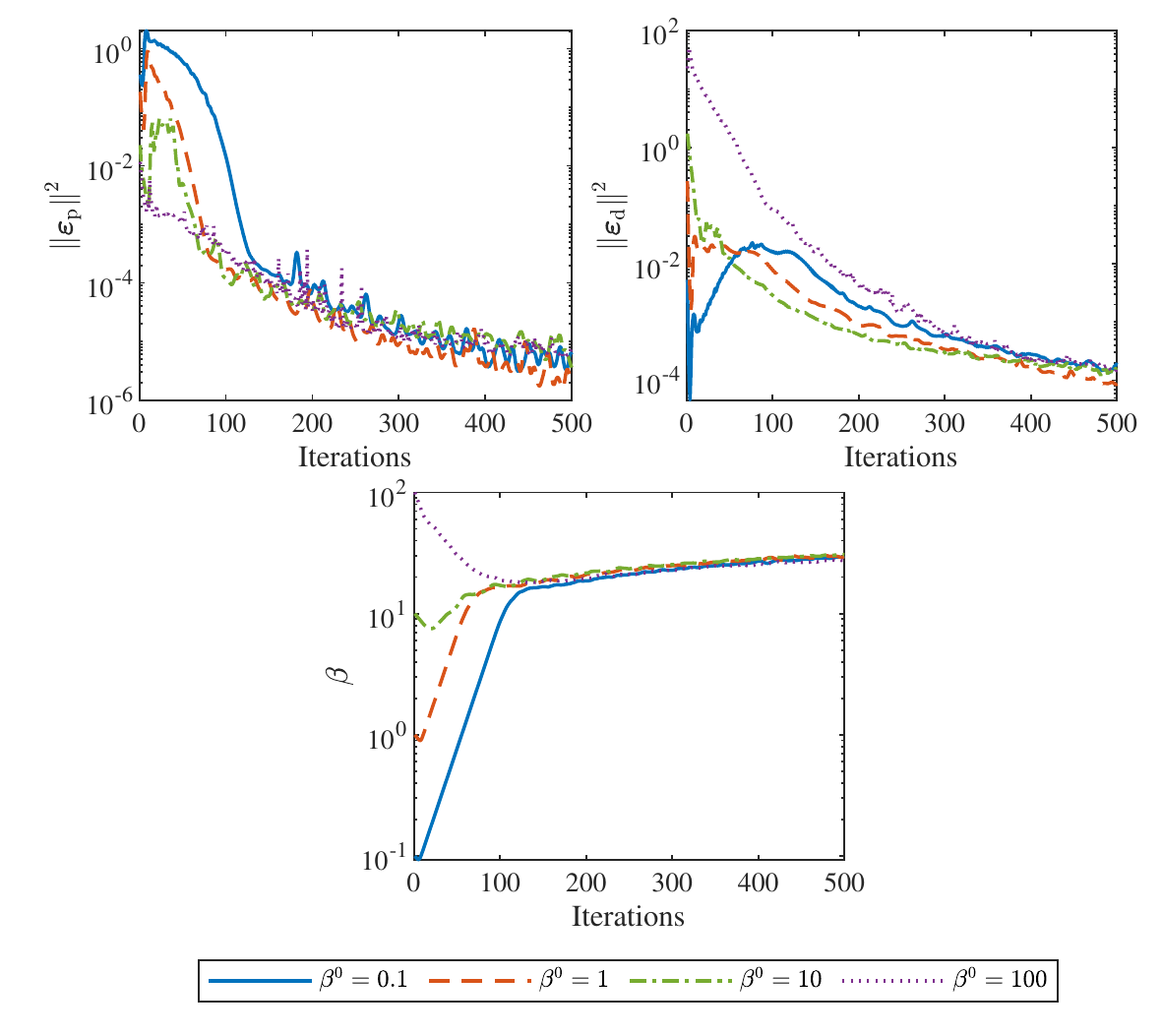}
\caption{The average squared norm of primal residual $\boldsymbol{\varepsilon}_\mathrm{p}$ and dual residual $\boldsymbol{\varepsilon}_\mathrm{d}$, and penalty parameter $\beta^k$ in the proposed self-adaptive updating strategy together with Anderson acceleration under different $\beta^0$ in 100 Monte Carlo runs.}
\label{fig:adap}
\end{figure}

\section{Conclusion}
In this work, a novel self-adaptive strategy for penalty updating was proposed to accelerate the nonconvex ADMM algorithm used in low-order model identification. In comparison with generic penalty updating strategy, the proposed method gives rise to a more balanced suppression of both primal and dual residuals, and thus exhibits better convergence performance. The implementation with Anderson acceleration was developed to attain further improvement. Numerical studies demonstrated the significant accelerations brought by the proposed method and its insensitivity to initial values. The theoretical convergence properties of the proposed method and its applications to more general nonconvex programs are worthy of future investigation.  

\addtolength{\textheight}{-12cm}   







\bibliographystyle{IEEEtran}
\bibliography{ref_bib}

\end{document}